\documentclass[12pt,oneside]{amsart}
\usepackage[margin=1in]{geometry}

\usepackage{graphicx}
\usepackage{amssymb}
\usepackage{tikz}
\usepackage{caption}
\usepackage{subcaption}
\usepackage{amsmath,amsthm,amscd,amssymb}
\usepackage{latexsym}
\usepackage[colorlinks,citecolor=red,pagebackref,hypertexnames=false]{hyperref}
\hypersetup{%
  colorlinks = true,
  citecolor=red,
  linkcolor  = black
}

\usepackage{ulem}
\usepackage{soul}
\numberwithin{equation}{section}
\theoremstyle{plain}
\newtheorem{theorem}{Theorem}[section]
\newtheorem{lemma}[theorem]{Lemma}
\newtheorem{corollary}[theorem]{Corollary}

\theoremstyle{definition}
\newtheorem{definition}[theorem]{Definition}

\theoremstyle{remark}
\newtheorem{remark}[theorem]{Remark}

\newtheorem{case[theorem]}{Case}

\date{\today}      

\begin{document} 
\title{Interpolation of point configurations in the discrete plane} 

\author{Esen Aksoy} 
\address{Department of Mathematics \\ Koc University}
\email{eaksoyzc@gmail.com} 

\author{Alex Iosevich} 
\address{Department of Mathematics \\ University of Rochester} 
\email{iosevich@gmail.com} 

\author{Brian McDonald}
\address{Department of Mathematics \\ University of Georgia}
\email{brian.mcdonald@uga.edu}

\thanks{The second listed author was supported in part by the National Science Foundation under grant no. 2154232. The second listed author wishes to thank the Isaac Newton Institute (INI) in Cambridge, Great Britain, for their support and hospitality. This work was done in part during the INI program, entitled "Multivariate approximation, discretization, and sampling recovery".}


\begin{abstract}
Defining distances over finite fields formally by $||x-y||:=(x_1-y_1)^2+\cdots + (x_d-y_d)^2$ for $x,y\in \mathbb{F}_q^d$, distance problems naturally arise in analogy to those studied by Erd\H{o}s and Falconer in Euclidean space.  Given a graph $G$ and a set $E\subseteq \mathbb{F}_q^2$, let $\Delta_G(E)$ be the generalized distance set corresponding to $G$.  In the case when $G$ is the complete graph on $k+1$ vertices, Bennett, Hart, Iosevich, Pakianathan, and Rudnev showed that when $|E|\geq q^{d-\frac{d-1}{k+1}}$, it follows that $|\Delta_G(E)|\geq cq^{\binom{k+1}{2}}$.  In the case when $k=d=2$, the threshold can be improved to $|E|\geq q^{\frac{8}{5}}$.  
\\
\\
Moreover, Jardine, Iosevich, and McDonald showed that in the case when $G$ is a tree with $k+1$ vertices, then whenever $E\subseteq \mathbb{F}_q^d$, $d\geq 2$ satisfies $|E|\geq C_kq^{\frac{d+1}{2}}$, it follows that $\Delta_G(E)=\mathbb{F}_q^k$.  In this paper, we present a technique which enables us to study certain graphs with both rigid and non-rigid components.  In particular, we show that for $E\subseteq \mathbb{F}_q^2$, $q=p^n$, $n$ odd, $p\equiv 3 \ \text{mod} \ 4$, and $G$ is the graph consisting of two triangles joined at a vertex, then whenever $|E|\geq q^{\frac{12}{7}}$, it follows that $|\Delta_G(E)|\geq cq^6$. The key to our approach is a configuration interpolation technique that allows us to trade off geometric complications.

\end{abstract}

\date{\today}
\maketitle
\section{Introduction} \label{s:intro}
\noindent Let $\mathbb{F}_q^d$ be the $d$-dimensional vector space over the finite field with $q=p^n$ elements, for an odd prime $p$.  We adopt the notion of distance defined by
$$ ||x-y||=(x_1-y_1)^2+\cdots +(x_d-y_d)^2, $$
for $x,y\in \mathbb{F}_q^d$.  

We note that this is not a metric, and $||\cdot||$ is not a norm.  Our ``distances'' are elements of $\mathbb{F}_q$.  Nevertheless, this notion of distance shares geometric, algebraic, and analytic properties with the usual Euclidean distance. For example, if $|| \theta x||=||x||$ if $\theta \in O_2({\mathbb F}_q)$, the group of orthonormal two-by-two matrices over ${\mathbb F}_q$. As such, there are often analogies between distance-related problems over finite fields and their Euclidean analogs. 

One of the distance-type results in ${\mathbb F}_q^d$ that has received much attention in the past decade is the Erdos-Falconer distance problem formulated by the second listed author of this paper and Misha Rudnev in \cite{IR07}. The question they posed is, how large does $E \subset {\mathbb F}_q^d$, $d \ge 2$ needs to be to ensure that $\Delta(E)={\mathbb F}_q$, or, more conservatively, that $|\Delta(E)| \ge \frac{q}{2}$, where here, and throughout, 
$$ \Delta(E)=\left\{||x-y||: x,y \in E \right\}.$$

In \cite{IR07}, the authors proved that if $|E| \ge 3q^{\frac{d+1}{2}}$, then $\Delta(E)={\mathbb F}_q$. In \cite{HIKR11} it was shown that the exponent $\frac{d+1}{2}$ cannot be improved in odd dimensions even if only wish to show that $|\Delta(E)| \ge \frac{q}{2}$. In dimension two, the critical exponent $\frac{4}{3}$ was established in \cite{BHIPR17}. If $q$ is prime, a further improvement down to $\frac{5}{4}$ was shown by Murphy, Petridis, Pham, Rudnev, and Stevens in \cite{MPPRS20}. In higher even dimension, the exponent $\frac{d+1}{2}$ is still the best known and we do not know if it can be improved. 

It is interesting to note that the Erdos-Falconer distance problem in ${\mathbb F}_q^d$ is an arithmetic analog of the Falconer distance problem in Euclidean space. This problem, posed by Falconer in 1986 asks for the smallest threshold $s_0$ such that if the Hausdorff dimension of a compact set $E$ is greater than $s_0$, then the Lebesgue measure of the distance set $\Delta(E)=\{|x-y|: x,y \in E\}$ is positive. Here $|\cdot|$ is the standard Euclidean distance. See, for example, for the latest developments on this problem. 

An interesting analog of the Erdos-Falconer distance problem in ${\mathbb F}_q^d$ arises if instead of considering distances determined by pairs of points, we consider a variety of relationships between the elements of configurations involving multiple points. More precisely, 	Let $G$ be a connected graph on $k+1$ vertices. Let $V=\{x^1,x^2, \dots, x^{k+1}\}$ denote the vertex set and $e_G$ the edge map, where $e_G(i,j)=1$ if $x^i$ and $x^j$ are connected by an edge, and $0$ otherwise. We will only consider undirected graphs with no self-edges, so $e_G(i,i)=0$ and $e_G(i,j)=e_G(j,i)$ for all $i,j$. Let ${\mathcal E}(G)$ denote the edge set, namely 
$$ \left\{(i,j) \in V \times V: e(x^i,x^j)=1 \right\}/\sim,$$ where $\sim$ is the equivalence relation $(i,j)\sim (j,i)$.
	
\begin{definition} Given such a graph, $E \subset {\mathbb F}_q^d$, $d \ge 2$, and a set of positive real numbers 
${\{t_{ij}\}}_{(i,j) \in {\mathcal E}(G)}$, let $\nu_G$ be defined by the relation 
$$ \sum f(\vec{t}) \nu_G(\vec{t})=\sum \dots \sum f(D_G(x^1, \dots, x^{k+1})) E(x^1)E(x^2) \dots E(x^{k+1}),$$ where 
$D_G(x^1, \dots, x^{k+1})$ is a vector of length equal to $|{\mathcal E}(G)|$ with entries $||x^i-x^j||$, where $(i,j) \in {\mathcal E}(G)$, the entries are in the dictionary order. 

\vskip.125in 

Define $\Delta_G(E)$ to be the support of the function $\nu_G$. We refer to this as the generalized distance set of $E\subseteq \mathbb{F}_q^d$ corresponding to the graph $G$.
\end{definition} 

\vskip.125in 

Observe that if $k=1$ and $G$ is the complete graph on two vertices, then $\Delta_G(E)=\Delta(E)$, the distance set defined above. If $k=2$ and $G$ is a complete graph on three vertices, then $\Delta_G(E)$ essentially counts the number of congruence classes of triangles determined by $E$. In the case when $G$ is a complete graph on $k+1$ vertices, it is shown in \cite{BHIPR17} that if $|E| \ge q^{d-\frac{d-1}{k+1}}$, then $|\Delta(E)| \ge cq^{k+1 \choose 2}$. Note that ${k+1 \choose 2}$ is precisely the size of the edge set of $G$. The threshold $d-\frac{d-1}{k+1}$ can be improved to $\frac{8}{5}$ when $k=d=2$. 

\vskip.125in 

Another result relevant to our narrative was established in \cite{IJM2021} where the authors showed that if $G$ is any connected tree graph on $k+1$ vertices, then if $E \subset {\mathbb F}_q^d$, $d \ge 2$, satisfies $|E| \ge C_kq^{\frac{d+1}{2}}$, then $\Delta_G(E)={\mathbb F}_q^k$. 

\vskip.125in 

The key difference between the results for complete graphs and those for tree graphs is that in the former case we can take advantage of rigidity, while in the latter case, we can take advantage of the inductive nature of the structure of the graph. In this paper, we take a step in the direction of combining these concepts by considering a graph in ${\mathbb F}_q^2$ formed by joining two complete graphs on three vertices at a common vertex. We believe that the techniques we develop here can be broadly generalized to handle more complicated configurations. 

\vskip.125in 

{\bf Notation:} We will use the asymptotic notation $A\lesssim B$ to mean that for $q$ sufficiently large, $|A|\lesssim c|B|$, where $c$ is a constant independent of $q$.  We will use this interchangeably with Bachmann-Landau notation, i.e. $A\lesssim B$ is equivalent to $A=O(B)$, which is equivalent to $B=\Omega(A)$.  We will use $|S|$ to denote the cardinality of a finite set $S$, and will often identify a set with its indicator function for convenience, i.e. $S(x)=1$ when $x\in S$ and $S(x)=0$ otherwise.

Let 
$$
S_t=\{x\in \mathbb{F}_q^d: ||x||=t\}
$$
denote the sphere of radius $t$ centered at the origin, so that $S_t(x)=1$ when $||x||=t$.

\begin{definition}
Let $B$ be the ``bow tie" graph consisting of two triangles joined at a vertex.
    
\end{definition}

\begin{figure}[h!]
    \centering
    \begin{tikzpicture}
        \coordinate (y) at (-1.73,1);
        \coordinate (z) at (-1.73,-1);
        \coordinate (x) at (0,0); 
        \coordinate (u) at (1.73,1);
        \coordinate (v) at (1.73,-1);

        \draw[thick] (y) -- (z) -- (x) -- cycle;

        \draw[thick] (x) -- (u) -- (v) -- cycle;

        \fill (y) circle (2pt);
        \fill (z) circle (2pt);
        \fill (x) circle (2pt);
        \fill (u) circle (2pt);
        \fill (v) circle (2pt);

        \node[above left] at (y) {$y$};
        \node[below left] at (z) {$z$};
        \node[above] at (x) {$x$};
        \node[above right] at (u) {$u$};
        \node[below right] at (v) {$v$};
    \end{tikzpicture}
    \caption{The bow tie graph $B$}
    \label{fig:triangles}
\end{figure}
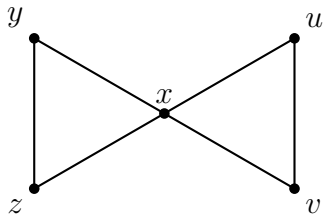

\begin{figure}[h!]
    \centering
    \begin{tikzpicture}
        \coordinate (X) at (0,0);
        \coordinate (Y) at (2,0);
        \coordinate (Z) at (-2,0);
        \coordinate (U) at (1,1.73);
        \coordinate (V) at (3,1.73);

        \draw[thick] (X) -- (Y);
        \draw[thick] (Y) -- (V);
        \draw[thick] (U) -- (V);
        \draw[thick] (U) -- (X);
        \draw[thick] (X) -- (Z);
        \draw[thick] (U) -- (Y);

        \fill (X) circle (2pt);
        \fill (Y) circle (2pt);
        \fill (Z) circle (2pt);
        \fill (U) circle (2pt);
        \fill (V) circle (2pt);

        \node[below left] at (X) {$X$};
        \node[below right] at (Y) {$Y$};
        \node[below] at (Z) {$Z$};
        \node[above] at (U) {$U$};
        \node[below right] at (V) {$V$};
    \end{tikzpicture}
    \caption{The kite graph $K$ used in the proof of Theorem \ref{mainfieldthm}}
    \label{fig:graph}
\end{figure}
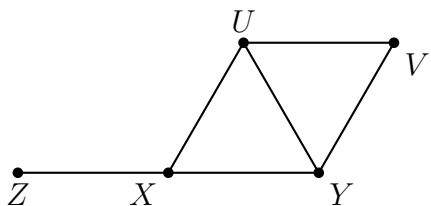

\vskip.125in 

\begin{theorem} \label{mainfieldthm} Let $E \subset {\mathbb F}_q^2$, where $q=p^n$, $n$ odd, $p \equiv 3 \mod 4$. Let $G$ be a bow-tie graph defined above. Suppose that 
$|E| \ge q^{\frac{12}{7}}$. Then there exists $c>0$, independent of $q$ such that $|\Delta_G(E)| \ge cq^6$. 
\end{theorem}

\begin{remark}
The condition that $q=p^n$ for $n$ odd and $p\equiv 3 \ \text{mod} \ 4$ was chosen precisely so that $x^2+1=0$ has no solutions in $\mathbb{F}_q$.  This ensures that there are no non-trivial circles of radius zero.  Indeed, if nonzero $x=(x_1,x_2)\in \mathbb{F}_q^2$ satisfies $||x||=0$, then $x_1^2+x_2^2=0$.  Assuming, without loss of generality, that $x_2\neq 0$, it follows that $\frac{x_1}{x_2}$ is a solution to $x^2+1=0$.

\end{remark}

\section{The generalized distance set for the triangle}\label{s:triangle}

\noindent For a subset $E\subseteq \mathbb{F}_q^2$, let $\Delta_T(E)$ denote the generalized distance set of $E$ corresponding to a triangle, i.e.
$$
\Delta_T(E):=\left\{t=(t_1,t_2,t_3)\in \mathbb{F}_q^3: \ \text{For some} \ x,y,z\in E, \ ||x-y||=t_1, \ ||x-z||=t_2, \ ||y-z||=t_3\right\}
$$
If $\Delta_T(E)=\mathbb{F}_q^3$, this means that every possible congruence class of triangles is realized by vertices in $E$.  We will investigate how large $E$ must be to ensure a positive proportion of all possible congruence classes, i.e. 
$$
\Delta_T(E)\gtrsim q^3.
$$
This has already been studied by \cite{BHIPR17}, obtaining the following result:
\begin{theorem}\label{BHIPR}[Bennett, Hart, Iosevich, Pakianathan, and Rudnev (2017)]
If $|E|\geq Cq^{\frac{8}{5}}$, then $\Delta_T(E)\geq cq^3$.
\end{theorem}
\vskip.125in
We present a slightly modified version of their proof for two reasons:  The approach via group actions is similar to how we will begin our proof of Theorem \ref{mainfieldthm}, and also many of the lemmas from this section will be needed in the proof of Theorem \ref{mainfieldthm}.
\vskip.125in
\begin{remark}
For $\nu_G$ as defined in Section \ref{s:intro}, for any $t\in \mathbb{F}_q^3$, we have
$$
\nu_T(t):=\left|\left\{(x,y,z)\in E^3: ||x-y||=t_1, \ ||x-z||=t_2, \ ||y-z||=t_3\right\}\right|
$$
\end{remark}

Since any tuple $(x,y,z)\in E^3$ will be realized by some $t\in \mathbb{F}_q^3$, we see that
$$
|E|^3=\sum_{t\in \mathbb{F}_q^3}\nu_T(t)
$$
Squaring, and applying Cauchy-Schwarz, we get
$$
|E|^6\leq |\Delta_T(E)|\sum_{t\in \mathbb{F}_q^3}\nu_T(t)^2,
$$
since $\Delta_T(E)$ is the set of $t\in \mathbb{F}_q^3$ for which $\nu_T(t)$ is non-zero.  Thus, finding a lower bound for $|\Delta_T(E)|$ is reduced to finding an upper bound for $||\nu_T||_{L^2}$. To this end, we express $||\nu_T||_{L^2}$ as a sum over the orthogonal group, taking advantage of the fact that distances are characterized by the action of the orthogonal group on $\mathbb{F}_q^2$.  
\begin{lemma}\label{group action}
$$
\sum_{t\in \mathbb{F}_q^3}\nu_T(t)^2
\lesssim\sum_{\theta\in O(\mathbb{F}_q^2)}\sum_{w\in \mathbb{F}_q^2}\lambda_\theta(w)^3,
$$
where $O(\mathbb{F}_q^2)$ is the orthogonal group, and 
$$
\lambda_{\theta}(w):=|\{(u,v)\in E^2: u-\theta v=w\}|.
$$
\end{lemma}

\begin{proof}
$$
\nu_T(t)=\sum_{x,y,z\in \mathbb{F}_q^d}E(x)E(y)E(z)S_{t_1}(x-y)S_{t_2}(x-z)S_{t_3}(y-z), 
$$
and so

$$
\sum_{t\in \mathbb{F}_q^3}\nu_T(t)^2
=\sum_t\sum_{\substack{x,y,z\in E \\ x',y',z'\in E}}S_{t_1}(x-y)S_{t_2}(x-z)S_{t_3}(y-z)S_{t_1}(x'-y')S_{t_2}(x'-z')S_{t_3}(y'-z')
$$
$$
=\left|\left\{(x,y,z,x',y',z')\in E^6:||x-y||=||x'-y'||, \  ||x-z||=||x'-z'||, \  ||y-z||=||y'-z'||\right\}\right|
$$
Note that $||x-y||=||x'-y'||$ if and only if there is an orthogonal transformation $\theta$ such that $x-y=\theta(x'-y')$, or equivalently, $x-\theta x'=y-\theta y'$.  The choice of $\theta$ is unique up to reflection, as long as the lengths are nonzero.  We can make this same observation for $||x-z||=||x'-z'||$ and for $||y-z||=||y'-z'||$.  Moreover, the same choice of $\theta$ will work for all three cases, since $(x,y,z)$ and $(x',y',z')$ form congruent triangles, meaning there is an orthogonal transformation mapping $x-y\mapsto x'-y'$, $x-z\mapsto x'-z'$, and $y-z\mapsto y'-z'$.  Thus,
$$
\sum_{t\in \mathbb{F}_q^3}\nu_T(t)^2
\leq\sum_{\theta\in O(\mathbb{F}_q^2)}\sum_{w\in \mathbb{F}_q^2}\lambda_\theta(w)^3,
$$
Where 
$$
\lambda_{\theta}(w)=\left|\left\{u,v\in E: u-\theta v=w\right\}\right|.
$$
We have an inequality instead of an identity because of the case when one or more of the components of $t$ are zero.  
\end{proof}

Thus, Theorem \ref{BHIPR} has been reduced to finding an appropriate upper bound for $\sum_{\theta,w}\lambda_{\theta}(w)^3$.  
Note that $|\lambda_{\theta}(w)|\leq |E|$ everywhere, since once $\theta,v,w$ are fixed, $u=w+\theta v$ is determined.  Therefore,
$$
\sum_{\theta,w}\lambda_{\theta}(w)^3
=\sum_{\theta,w}\lambda_{\theta}(w)\left(\lambda_{\theta}(w)-\frac{|E|^2}{q^2}\right)^2
+2\frac{|E|^2}{q^2}\sum_{\theta,w}\lambda_{\theta}(w)^2
-\frac{|E|^4}{q^4}\sum_{\theta,w}\lambda_{\theta}(w)
$$
$$
\lesssim|E|\sum_{\theta,w}\left(\lambda_{\theta}(w)-\frac{|E|^2}{q^2}\right)^2+\frac{|E|^2}{q^2}\sum_{\theta,w}\lambda_{\theta}(w)^2
$$
\vskip.125in
In this calculation, we have subtracted the 0-th Fourier coefficient, 
$$
\widehat{\lambda_{\theta}}(0)=\frac{|E|^2}{q^2}.
$$
We can bound both of these terms in essentially the same way.
\vskip.125in

\begin{lemma}\label{bound1}
$$
\sum_{\theta,w}\left(\lambda_{\theta}(w)-\frac{|E|^2}{q^2}\right)^2\lesssim q|E|^{\frac{5}{2}}, 
\ \ \ 
\text{and}
\ \ \ 
\sum_{\theta,w}\lambda_{\theta}(w)^2\lesssim q|E|^{\frac{5}{2}}+\frac{|E|^4}{q}
$$
\end{lemma}

\vskip.125in

\begin{corollary}\label{Triangle_corollary}
If $|E|\geq q^{\frac{8}{5}}$, then
$$
\sum_{\theta,w}\lambda_{\theta}(w)^3
\lesssim \frac{|E|^6}{q^3},
$$
and thus the conclusion of Theorem \ref{BHIPR} follows.
    
\end{corollary}

\vskip.125in
We now prove the Lemma, completing our analysis of the generalized distance set corresponding to the triangle.

\begin{proof}

Let $\gamma_{\theta}:=\lambda_{\theta}-\frac{|E|^2}{q^2}$, so that $\widehat{\gamma_{\theta}}(0)=0$, and for $m\neq 0$,

$$
\widehat{\gamma_{\theta}}(m)=\widehat{\lambda_{\theta}}(m)
=q^{-2}\sum_{w\in \mathbb{F}_q^2}\chi(-m\cdot w)\lambda_{\theta}(w)
=q^{-2}\sum_{u,v\in \mathbb{F}_q^2}\chi(-m\cdot (u-\theta v))E(u)E(v)
$$
$$
=q^{2}\hat{E}(m)\hat{E}(\theta^{-1}m).
$$
Thus,
$$
\sum_{\theta,m}|\widehat{\gamma_{\theta}}(m)|^2
=q^{4}\sum_{\theta,m}|\hat{E}(m)|^2|\hat{E}(\theta^{-1}m)|^2
$$
Summing first in $\theta$, let $t=||m||$ and $g(\ell)=\overline{\hat{E}}$, so that
$$
\sum_{\theta}|\hat{E}(\theta^{-1}m)|^2
=2\sum_{||\ell||=t}|\hat{E}(\ell)|^2
=2\sum_{\ell}\hat{E}(\ell)S_t(\ell)g(\ell)
=2\sum_{\ell}E(\ell)\widehat{gS_t}(\ell)
$$
$$
\lesssim |E|^{\frac{3}{4}}||\widehat{gS_t}||_{L^4}.
$$
We can bound this quantity using the fact that whenever $m+\ell=m'+\ell'$, for $m,m',\ell,\ell'\in S_t$, one of the following must hold:
\begin{itemize}
\item $m=m'$, $\ell=\ell'$

\item $m=\ell'$, $\ell=m'$

\item $m+\ell=m'+\ell'=0$
    
\end{itemize}
Thus,
$$
||\widehat{gS_t}||_{L^4}^4
=q^{-8}\sum_{\xi}\left|\sum_{m}\chi(-m\cdot \xi)g(m)S_t(m)\right|^4
$$
$$
=q^{-8}\sum_{m,m',\ell,\ell'\in S_t}\sum_{\xi}\chi(-\xi\cdot (m+\ell-m'-\ell'))g(m)g(\ell)\overline{g(m')}\overline{g(\ell')}
$$
$$
=q^{-6}\sum_{\substack{m,m',\ell,\ell'\in S_t \\ m+\ell=m'+\ell'}}g(m)g(\ell)\overline{g(m')}\overline{g(\ell')}
$$
$$
\lesssim q^{-6}||gS_t||_{L^2}^4
$$
The last inequality follows by splitting the sum into the 3 cases listed above.  In the case $m=m'$, $\ell=\ell'$, as well as in the case $m=\ell'$, $\ell=m'$, this is trivial.  In the case $m+\ell=m'+\ell'=0$, this follows from Cauchy-Schwarz.  This gives us a recursive inequality,
$$
\sum_{\theta}|\hat{E}(\theta^{-1}m)|^2
=2||gS_t||_{L^2}^2
\lesssim q^{-\frac{3}{2}}|E|^{\frac{3}{4}}||gS_t||_{L^2},
$$
and thus,
$$
||gS_t||_{L^2}\lesssim q^{-\frac{3}{2}}|E|^{\frac{3}{4}},
$$
so
$$
\sum_{\theta}|\hat{E}(\theta^{-1}m)|^2\lesssim q^{-3}|E|^{\frac{3}{2}}.
$$
Finally, since
$$
\sum_m|\hat{E}(m)|^2=q^{-2}|E|,
$$
it follows that
$$
\sum_{\theta,w}\left(\lambda_{\theta}(w)-\frac{|E|^2}{q^2}\right)^2
=q^2\sum_{\theta,m}|\widehat{\gamma_{\theta}}(m)|^2
=q^6\sum_{\theta,m}|\hat{E}(m)|^2|\hat{E}(\theta^{-1}m)|^2
$$
$$
\lesssim q|E|^{\frac{5}{2}}.
$$
Adding back in the 0-th Fourier coefficient, we see that
$$
\sum_{\theta,w}\lambda_{\theta}(w)^2
=q^2\sum_{\theta,m}|\widehat{\lambda_{\theta}}(m)|^2
=\frac{|E|^4}{q}+q^2\sum_{\theta,m}|\widehat{\gamma_{\theta}}(m)|^2
$$
$$
\lesssim \frac{|E|^4}{q}+q|E|^{\frac{5}{2}}.
$$

\end{proof}

With these techniques in hand, we are ready to move on to discuss the bow tie graph.

\section{Proof of Theorem \ref{mainfieldthm}}
\subsection{Interpolation scheme for generalized distance graphs.}
To understand $\Delta_B(E)$, where $B$ is the bow tie graph as defined in Section \ref{s:intro}, it suffices to understand $\Delta_K(E)$, where $K$ is the kite graph defined below.

\vskip.125in

\begin{definition}
Let $K$ be the ``kite" graph consisting of two triangles sharing an edge, with an extra edge coming from one of the vertices not on the shared edge of the two triangles.
    
\end{definition}

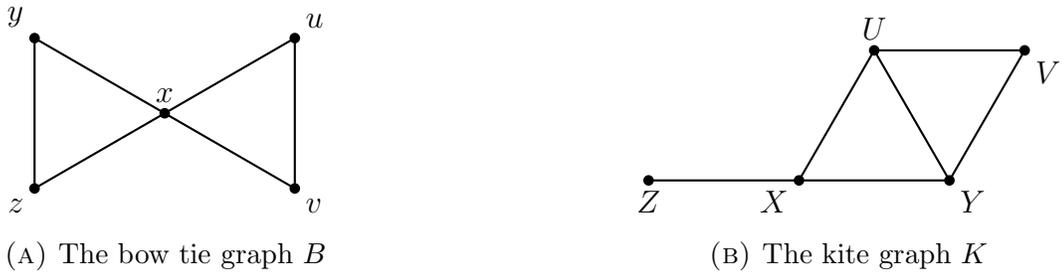
\begin{figure}[h!]
    \centering
    \begin{subfigure}[b]{0.45\textwidth}
        \centering
        \begin{tikzpicture}
            \coordinate (y) at (-1.73,1);
            \coordinate (z) at (-1.73,-1);
            \coordinate (x) at (0,0); 
            \coordinate (u) at (1.73,1);
            \coordinate (v) at (1.73,-1);

            \draw[thick] (y) -- (z) -- (x) -- cycle;

            \draw[thick] (x) -- (u) -- (v) -- cycle;

            \fill (y) circle (2pt);
            \fill (z) circle (2pt);
            \fill (x) circle (2pt);
            \fill (u) circle (2pt);
            \fill (v) circle (2pt);

            \node[above left] at (y) {$y$};
            \node[below left] at (z) {$z$};
            \node[above] at (x) {$x$};
            \node[above right] at (u) {$u$};
            \node[below right] at (v) {$v$};
        \end{tikzpicture}
        \caption{The bow tie graph $B$}
        \label{fig:triangles}
    \end{subfigure}
    \hfill
    \begin{subfigure}[b]{0.45\textwidth}
        \centering
        \begin{tikzpicture}
            \coordinate (X) at (0,0);
            \coordinate (Y) at (2,0);
            \coordinate (Z) at (-2,0);
            \coordinate (U) at (1,1.73);
            \coordinate (V) at (3,1.73);

            \draw[thick] (X) -- (Y);
            \draw[thick] (Y) -- (V);
            \draw[thick] (U) -- (V);
            \draw[thick] (U) -- (X);
            \draw[thick] (X) -- (Z);
            \draw[thick] (U) -- (Y);

            \fill (X) circle (2pt);
            \fill (Y) circle (2pt);
            \fill (Z) circle (2pt);
            \fill (U) circle (2pt);
            \fill (V) circle (2pt);

            \node[below left] at (X) {$X$};
            \node[below right] at (Y) {$Y$};
            \node[below] at (Z) {$Z$};
            \node[above] at (U) {$U$};
            \node[below right] at (V) {$V$};
        \end{tikzpicture}
        \caption{The kite graph $K$}
        \label{fig:graph}
    \end{subfigure}
    \caption{To study $\Delta_B(E)$ for $E\subseteq \mathbb{F}_q^2$, we will instead study $\Delta_K(E)$, obtaining a bound for $|\Delta_B(E)|$ by the Hadamard three-lines theorem.}
    \label{figure_both}
\end{figure}

\vskip.125in

\begin{definition}
For $a,b\in \mathbb{C}$, $E\subseteq \mathbb{F}_q^2$, let
$$
\Psi(a,b):=\sum_{\substack{x,x'\in \mathbb{F}_q^2 \\ \theta,\phi\in O(\mathbb{F}_q^2)}}\lambda_{\theta}(x-\theta x')^a\lambda_{\phi}(x-\phi x')^b.
$$ 
\end{definition}

\vskip.125in

\begin{lemma}
Let $E\subseteq \mathbb{F}_q^2$, and let $\nu_B$ be as defined in Section \ref{s:intro}, for the bow tie graph $B$.  Then,  
$$
\sum_{t\in \mathbb{F}_q^6}\nu_B(t)^2\leq
\Psi(2,2).
$$

\end{lemma}

\vskip.125in

\begin{proof}
This is analogous to Lemma \ref{group action}.  
The left-hand sum counts pairs of tuples $$
(x,y,z,u,v),(x',y',z',u',v')\in E^5
$$
whose edges, interpreted as in Figure \ref{figure_both}, correspond to the same $t\in \mathbb{F}_q^6$.  $||x-y||=||x'-y'||$ if and only if there is some $\theta\in O(\mathbb{F}_q^2)$ with $x-\theta x' = y-\theta y'$.  The choice of $\theta$ is unique up to reflection as long as the lengths are nonzero, and because we have the triangle $(x,y,z)$ contained in the graph $B$, it will be the same rotation for all three edges, i.e. there exists one $\theta\in O(\mathbb{F}_q^2)$ with 
$$
x-\theta x'=y-\theta y'=z-\theta z'.
$$
Similarly, there will be one $\phi\in O(\mathbb{F}_q^2)$ with 
$$
x-\phi x'=u-\phi u'=v-\phi v'.
$$
Therefore, 
$$
\sum_{t\in \mathbb{F}_q^6}\nu_B(t)^2
\leq\Psi(2,2),
$$
where once again the inequality comes from the case when one or more components of $t$ are zero.  
\end{proof}
\vskip.125in
We are now ready to use the Hadamard three-lines theorem to reveal how the kite graph $K$ is connected to this.
\vskip.125in
\begin{lemma}
For any $E\subseteq \mathbb{F}_q^2$,
$$
\Psi(2,2)\leq \Psi(3,1).
$$
    
\end{lemma}

\vskip.125in

\begin{proof}
First, let $\psi(\sigma)=\Psi(2+\sigma,2-\sigma)$.  Then, since $\lambda_{\theta}$ is always real and non-negative, it follows that
$$
|\psi(\sigma)|
\leq \psi(\Re(\sigma)).  
$$
By the Hadamard three-lines theorem, it then follows that
$$
\Psi(2,2)=\psi(0)\leq \sqrt{\psi(-1)\psi(1)}=\Psi(3,1).
$$
    
\end{proof}

\vskip.125in

The kite graph $K$ corresponds to $\Psi(3,1)$ in the sense that the bow tie graph $B$ corresponds to $\Psi(2,2)$.  We do not need to directly refer to $K$ in order to bound $\Psi(3,1)$ and hence to bound $\Psi(2,2)$ and to understand $\Delta_B(E)$, but placing it in this context is useful heuristically.  By a similar reasoning with larger graphs with rigid components, one might hope to build a generalized-distance-graph calculus, understanding $\Delta_G(E)$ for a family of graphs $G$ by an interpolation scheme involving less complicated graphs.

\subsection{Estimating $\Psi(3,1)$ corresponding to the kite graph $K$.}
We want to estimate the sum
$$
\Psi(3,1)=\sum_{\substack{x,x'\in \mathbb{F}_q^2 \\ \theta,\phi\in O(\mathbb{F}_q^2)}}\lambda_{\theta}^3(x-\theta x')\lambda_{\phi}(x-\phi x')
$$
It will be helpful in our analysis to remove the $\lambda_{\phi}$ term, by instead only considering the case where $||x-z||=a$ for a fixed $a\in \mathbb{F}_q$, and then later summing in $a$.  For $t_6=a$ fixed, but $t_1,...,t_5$ varying, the number of pairs $(u,v,x,y,z),(u',v',x',y',z')$ corresponding to the same $t\in \mathbb{F}_q^6$ is given by
$$
\sum_{\theta\in O(\mathbb{F}_q^2)}\sum_{x,x'\in \mathbb{F}_q^2}\lambda_{\theta}^3(x-\theta x')f(x)f(x'),
$$
where
$$
f(x)=E(x)\cdot E\ast S_a(x).
$$
\begin{definition}
Let $f(x)=E(x)\cdot E\ast S_a(x)$.  We define $\alpha_{\theta}(w)$, a weighted count of pairs $x,x'\in E$ satisfying $x-\theta x'=w$, by
$$
\sum_{w}g(w)\alpha_{\theta}(w)
=\sum_{x,x'\in \mathbb{F}_q^2}g(x-\theta x')f(x)f(x').
$$
for an arbitrary function $g$.  Equivalently,
$$
\alpha_{\theta}(w):=\sum_{\substack{x,x'\in \mathbb{F}_q^2 \\ x-\theta x'=w }}f(x)f(x').
$$
\end{definition}

\vskip.125in

\begin{lemma}\label{properties_alpha}
We collect some information about $\alpha_{\theta}$ which will be useful later.
\begin{enumerate}
\item
$$
\widehat{\alpha_{\theta}}(m)=q^2\hat{f}(m)\overline{\hat{f}(\theta^{-1}m)},   \ \textup{and in particular},
$$
$$
\widehat{\alpha_{\theta}}(0)\approx \frac{|E|^4}{q^4} \ \textup{if} \ |E|\geq Cq^{\frac{3}{2}}.
$$
\item 
$$
||\alpha_{\theta}||_{L^{\infty}}\lesssim |E|^2
$$
\end{enumerate}

\end{lemma}

\begin{proof}
(1) Applying the definition of $\alpha_{\theta}$ to $g(w)=\chi(-m\cdot w)$, we see that
$$
\widehat{\alpha_{\theta}}(m)
=q^{-2}\sum_w\chi(-m\cdot w)\alpha_{\theta}(w)
=q^{-2}\sum_{x,x'}\chi(-m\cdot (x-\theta x'))f(x)f(x')
$$
$$
=q^2\hat{f}(m)\overline{\hat{f}(\theta^{-1} m)}
$$
Moreover, whenever $|E|\geq Cq^{\frac{3}{2}}$,
$$
\hat{f}(0)=q^{-2}||f||_{L^1}\approx \frac{|E|^2}{q^3},
$$
as shown in \cite{IR07}.  It follows that $\widehat{\alpha_{\theta}}(0)\approx \frac{|E|^4}{q^4}$.
\\
\\
(2)  It is clear from the definition that $f(x)\lesssim q$, since there are at most $q+1$ points on a circle over $\mathbb{F}_q$.  Thus,
$$
\alpha_{\theta}(w)=\sum_{\substack{x,x'\in \mathbb{F}_q^2 \\ x-\theta x'=w}}f(x)f(x')
\leq q\sum_{\substack{x,x'\in \mathbb{F}_q^2 \\ x-\theta x'=w}}f(x).
$$
For each $x$, there is exactly one $x'\in \mathbb{F}_q^2$ satisfying $x-\theta x'=w$.  Thus,
$$
\alpha_{\theta}(w)\leq q\sum_xf(x)\approx |E|^2,
$$
again by \cite{IR07}.  
    
\end{proof}

\vskip.125in

By the definition of $\alpha_{\theta}$, we see that
\vskip.125in
\begin{align*}
\sum_{\theta\in O(\mathbb{F}_q^2)}&\sum_{x,x'\in \mathbb{F}_q^2}\lambda_{\theta}^3(x-\theta x')f(x)f(x')
=\sum_{\theta\in O(\mathbb{F}_q^2)}\sum_{w\in \mathbb{F}_q^2}
\lambda_{\theta}^3(w)\alpha_{\theta}(w) \\
=& \sum_{\theta,w}\lambda_{\theta}^2(w)\left(\lambda_{\theta}(w)-\frac{|E|^2}{q^2}\right)\left(\alpha_{\theta}(w)-\frac{|E|^4}{q^4}\right) \\
& \ +\frac{|E|^4}{q^4}\sum_{\theta,w}\lambda_{\theta}^2(w)\left(\lambda_{\theta}(w)-\frac{|E|^2}{q^2}\right) +\frac{|E|^2}{q^2}\sum_{\theta,w}\lambda_{\theta}^2(w)\alpha_{\theta}(w) \\
=& I+II+III
\end{align*}
We have reduced the proof of Theorem \ref{mainfieldthm} to finding appropriate upper bounds for $I$, $II$, and $III$. A few of these follow easily from previous calculations.

\vskip.125in

\begin{lemma}
For $I$ defined as in the above discussion,
$$
|I|\lesssim q^{-1}|E|^{\frac{13}{2}}.
$$
\end{lemma}

\vskip.125in

\begin{proof}
By Cauchy-Schwarz,
$$
I=\sum_{\theta,w}\lambda_{\theta}^2(w)\left(\lambda_{\theta}(w)-\frac{|E|^2}{q^2}\right)\left(\alpha_{\theta}(w)-\frac{|E|^4}{q^4}\right)
$$
$$
\leq |E|^2\left(\sum_{\theta,w}\left(\lambda_{\theta}(w)-\frac{|E|^2}{q^2}\right)^2\right)^{\frac{1}{2}}\left(\sum_{\theta,w}\left(\alpha_{\theta}(w)-\frac{|E|^4}{q^4}\right)^2\right)^{\frac{1}{2}}
$$
$$
\lesssim |E|^2\left(q|E|^{\frac{5}{2}}\right)^{\frac{1}{2}}\left(q^{-3}|E|^{\frac{13}{2}}\right)^{\frac{1}{2}}
=q^{-1}|E|^{\frac{13}{2}}
$$
\end{proof}

\vskip.125in

We introduce a lemma from \cite{BHIPR17}, based on a Taylor series expansion of $x^n$, in order to bound some of the other terms.

\vskip.125in

\begin{lemma}\label{Averaging lemma}
For any function $\phi:X\to \mathbb{R}_{\geq 0}$, where $X=\mathbb{F}_q^d$ or $X=(\mathbb{Z}/q\mathbb{Z})^d$, and for any $n\geq 2$, we have
$$
\sum_{x\in X}\phi^n(x)
\leq q^{-d(n-1)}||\phi||_{L^1}^n+\frac{n(n-1)}{2}||\phi||_{L^{\infty}}^{n-2}\sum_{x\in X}\left(\phi(x)-\frac{||\phi||_{L^1}}{q^d}\right)^2
$$

\end{lemma}

\vskip.125in

\begin{lemma}
If $|E|\geq q^{\frac{12}{7}}$, then
$$
|II|\lesssim q^{-6}|E|^{\frac{37}{4}}
$$
\end{lemma}

\vskip.125in

\begin{proof}
By Lemma \ref{Averaging lemma}, we see that
$$
\sum_{\theta,w}\lambda_{\theta}^4(w)
\leq \sum_{\theta}\frac{|E|^8}{q^6}+|E|^2\sum_{\theta,w}\left(\lambda_{\theta}(w)-\frac{|E|^2}{q^2}\right)^2
\lesssim \frac{|E|^8}{q^5}+q|E|^{\frac{9}{2}}
$$
$$
\lesssim \frac{|E|^8}{q^5},
$$
as long as $|E|\geq q^{\frac{12}{7}}$.  We use this, and Cauchy-Schwarz, to bound $II$.
$$
II=\frac{|E|^4}{q^4}\sum_{\theta,w}\lambda_{\theta}^2(w)\left(\lambda_{\theta}(w)-\frac{|E|^2}{q^2}\right)
$$
$$
\lesssim \frac{|E|^4}{q^4}\left(\sum_{\theta,w}\lambda_{\theta}^4(w)\right)^{\frac{1}{2}}\left(\sum_{\theta,w}\left(\lambda_{\theta}(w)-\frac{|E|^2}{q^2}\right)^2\right)^{\frac{1}{2}}
$$
$$
\lesssim \frac{|E|^4}{q^4}\cdot q^{-\frac{5}{2}}|E|^4\cdot \left(q|E|^{\frac{5}{2}}\right)^{\frac{1}{2}}
$$
$$
=q^{-6}|E|^{\frac{37}{4}}
$$
    
\end{proof}

\begin{lemma}
If $|E|\geq q^{\frac{12}{7}}$, then
$$
III\leq \frac{|E|^{10}}{q^7}.
$$
\end{lemma}

\vskip.125in

\begin{proof}
By Cauchy-Schwarz,
$$
III=2\frac{|E|^2}{q^2}\sum_{\theta,w}\lambda_{\theta}(w)^2\alpha_{\theta}(w)
\lesssim \frac{|E|^2}{q^2}\left(\sum_{\theta,w}\lambda_{\theta}(w)^4\right)^{\frac{1}{2}}\left(\sum_{\theta,w}\alpha_{\theta}(w)^2\right)^{\frac{1}{2}}.
$$
We can apply Lemma \ref{Averaging lemma} to both of these sums.
$$
\sum_{\theta,w}\lambda_{\theta}(w)^4
\lesssim \sum_{\theta}\frac{|E|^8}{q^6}+|E|^2\sum_{\theta,w}\left(\lambda_{\theta}(w)-\frac{|E|^2}{q^2}\right)^2
$$
$$
\lesssim \frac{|E|^8}{q^5}+q|E|^{\frac{9}{2}},
$$
by Lemma \ref{bound1}.  Similarly,
$$
\sum_{\theta,w}\alpha_{\theta}(w)^2
\lesssim \sum_{\theta}\frac{|E|^8}{q^6}+\sum_{\theta,w}\left(\alpha_{\theta}(w)-\frac{|E|^4}{q^4}\right)^2
$$
By Lemma \ref{properties_alpha},
$$
\sum_{\theta,w}(\alpha_{\theta}(w)-\widehat{\alpha_{\theta}}(0))^2
=q^6\sum_{\theta}\sum_{m\neq 0}|\hat{f}(m)|^2|\hat{f}(\theta^{-1} m)|^2
$$
Summing first in $\theta$, let $t=||m||$ and $h(\ell)=\overline{\hat{f}}$, so that
$$
\sum_{\theta}|\hat{f}(\theta^{-1} m)|^2
=2\sum_{||\ell||=t}|\hat{f}(\ell)|^2
=2\sum_{\ell}\hat{f}(\ell)S_t(\ell)h(\ell)
=2\sum_{\ell}f(\ell)\widehat{hS_t}(\ell)
$$
$$
\lesssim ||f||_{L^{\frac{4}{3}}}||\widehat{hS_t}||_{L^4}.
$$
From the same argument as in Lemma \ref{bound1}, we see that
$$
||\widehat{hS_t}||_{L^4}^4\lesssim q^{-6}||hS_t||_{L^2}^4.
$$
To bound $||f||_{L^{\frac{4}{3}}}$, we interpolate by Riesz-Thorin between $L^1$ and $L^2$, since $||f||_{L^1}$ and $||f||_{L^2}$ are known:
$$
||f||_{L^1}\sim \frac{|E|^2}{q}, \ \ \ ||f||_{L^2}\sim \frac{|E|^{\frac{3}{2}}}{q}.
$$
Therefore,
$$
||f||_{L^{\frac{4}{3}}}
\leq \left(\frac{|E|^2}{q}\right)^{\frac{1}{2}}\left(\frac{|E|^{\frac{3}{2}}}{q}\right)^{\frac{1}{2}}
=q^{-1}|E|^{\frac{7}{4}}.
$$
This yields a recursive inequality,
$$
\sum_{\theta}|\hat{f}(\theta^{-1} m)|^2=2||hS_t||_{L^2}^2
\lesssim q^{-\frac{5}{2}}|E|^{\frac{7}{4}}||hS_t||_{L^2},
$$
and so
$$
||hS_t||_{L^2}\lesssim q^{-\frac{5}{2}}|E|^{\frac{7}{4}},
$$
and
$$
\sum_{\theta}|\hat{f}(\theta^{-1} m)|^2
\lesssim q^{-5}|E|^{\frac{7}{2}}.
$$
Moreover,
$$
\sum_{m\in \mathbb{F}_q^2}|\hat{f}(m)|^2
=q^{-2}\sum_x{f(x)^2}
\approx \frac{|E|^3}{q^4}.
$$
Therefore,
$$
\sum_{\theta,w}(\alpha_{\theta}(w)-\widehat{\alpha_{\theta}}(0))^2
\lesssim q^6\cdot q^{-1}|E|^{\frac{7}{4}}\cdot q^{-\frac{3}{2}}\cdot q^{-\frac{5}{2}}|E|^{\frac{7}{4}}\cdot q^{-4}|E|^3
$$
$$
=q^{-3}|E|^{\frac{13}{2}}
$$ 
Putting together the above calculations, if $|E|\geq q^{\frac{12}{7}}$, then
$$
III
\lesssim \frac{|E|^2}{q^2}\left(\frac{|E|^8}{q^5}+q|E|^{\frac{9}{2}}\right)^{\frac{1}{2}}\left(\frac{|E|^8}{q^5}+q^{-3}|E|^{\frac{13}{2}}\right)^{\frac{1}{2}}\lesssim \frac{|E|^{10}}{q^7}.
$$
\end{proof}
Recall that these quantities depend on a parameter $a\in \mathbb{F}_q$, and we still need to sum in $a$, which will yield an extra factor of $q$ in each of these bounds.  Therefore, as long as $|E|\geq q^{\frac{12}{7}}$, 
$$
\Psi(3,1)
=\sum_{a}\sum_{\theta}\sum_{x,x'}\lambda_{\theta}^3(x-\theta x')f(x)f(x')
\lesssim \max\left(|E|^{\frac{13}{2}},q^{-5}|E|^{\frac{37}{4}},q^{-6}|E|^{10}\right).
$$
Finally, we see that
$$
\frac{|E|^{10}}{\Psi(3,1)}\geq cq^6
$$
as long as $|E|\geq Cq^{\frac{12}{7}}$.  This completes the proof of Theorem \ref{mainfieldthm}.




\newpage 

\begin{thebibliography}{99}

\bibitem{BHIPR17} M. Bennett, D. Hart, A. Iosevich, J. Pakianathan, and M. Rudnev, {\it Group actions and geometric combinatorics in ${\mathbb F}_q^d$}, Forum Math. \textbf{29} (2017), no. 1, 91-110.

\bibitem{BV12} Bourgain, J. (1-IASP-SM); Varj\'{u}, P. (1-PRIN) {\it Expansion in $SL_{d}(\Bbb{Z}/q\Bbb{Z})$; q arbitrary.} (English summary) Invent. Math. 188 (2012), no. 1, 151-173. \url{http://arxiv.org/pdf/1006.3365.pdf}

\bibitem{CHIO08} Chapman, J. Iosevich, A. {\it On rapid Generation of $SL_2({\Bbb F}_q)$ }, Integers Electronic Journal of Combinatorial Number Theory, Volume 9 (2009), 47-52. \url{http://www.emis.de/journals/INTEGERS/papers/j4/j4.Abstract.html}

\bibitem{HIKR11} D. Hart, A. Iosevich, D. Koh and M. Rudnev {\it Averages over hyperplanes, sum-product theory in vector spaces over finite fields and the Erd\H os-Falconer distance conjecture}, Transactions of the AMS, \textbf{363} (arXiv:0707.3473), (2011), 3255-3275. 

\bibitem{IR07} A. Iosevich and M. Rudnev {\it Erd\H os distance problem in vector spaces over finite fields}, Trans. Amer. Math. Soc. \textbf{359} (2007), no. 12, 6127-6142.

\bibitem{IJM2021} A. Iosevich, G. Jardine, and B. McDonald, {\it Cycles of arbitrary length in distance graphs on ${\mathbb F}_q^d$}. Tr. Mat. Inst. Steklova 314 (2021), pages 31-48. 

\bibitem{MPPRS20} B. Murphy, G. Petridis, T. Pham, M. Rudnev, and S. Stevens, {\it On the Pinned Distances Problem over Finite Fields}, arXiv:2003.00510, (2020). 

\bibitem{R18} M. Rudnev, {\it On the number of incidences between points and planes in three dimensions}, Combinatorica, \textbf{38} (2018), no.1, 219-254.

\end{thebibliography}
\end{document}